\documentclass[psamsfonts]{amsart}

\usepackage{amssymb,amsfonts}
\usepackage[all,arc]{xy}
\usepackage{enumerate}
\usepackage{mathrsfs}
\usepackage{mathtools}
\DeclarePairedDelimiterX\set[1]\lbrace\rbrace{\def\given{\;\delimsize\vert\;}#1}
\newtheorem{theorem}{Theorem}[section]
\newtheorem{corollary}[theorem]{Corollary}
\newtheorem{proposition}[theorem]{Proposition}
\newtheorem{lemma}[theorem]{Lemma}

\newtheorem{question}[theorem]{Question}

\theoremstyle{definition}
\newtheorem{definition}[theorem]{Definition}

\theoremstyle{remark}
\newtheorem{remark}[theorem]{Remark}

\let\phi=\varphi

\def\Min{\operatorname{Min}}

\let\oldbigwedge\bigwedge
\def\BIGwedge{{\textstyle\oldbigwedge}}
\def\medwedge{{\scriptstyle\oldbigwedge}}
\def\bigwedge{\mathchoice{\BIGwedge}{\BIGwedge}{\medwedge}{}}

\DeclareMathOperator{\Id}{Id}

\makeatletter

\let\epsilon=\varepsilon

\let\c@equation\c@theorem
\makeatother
\numberwithin{equation}{section}

\begin{document}
	\title[2-absorbing ideals of semirings]{On 2-absorbing ideals of commutative semirings}
	
	\author{\bfseries H. Behzadipour}
	
	\address{School of Electrical and Computer Engineering\\ University College of Engineering\\ University of Tehran\\ Tehran\\ Iran}
	
	\email{hussein.behzadipour@gmail.com}
	
	\author{\bfseries P. Nasehpour}
	
	\address{Department of Engineering Science\\ Golpayegan University of Technology\\ Golpayegan\\
		Iran}
	\email{nasehpour@gut.ac.ir, nasehpour@gmail.com}
	
	
	\begin{abstract}

In this paper, we investigate 2-absorbing ideals of commutative semirings and prove that if $\mathfrak{a}$ is a nonzero proper ideal of a subtractive valuation semiring $S$ then $\mathfrak{a}$ is a 2-absorbing ideal of $S$ if and only if $\mathfrak{a}=\mathfrak{p}$ or $\mathfrak{a}=\mathfrak{p}^2$ where $\mathfrak{p}=\sqrt\mathfrak{a}$ is a prime ideal of $S$. We also show that each 2-absorbing ideal of a subtractive semiring $S$ is prime if and only if the prime ideals of $S$ are comparable and if $\mathfrak{p}$ is a minimal prime over a 2-absorbing ideal $\mathfrak{a}$, then $\mathfrak{am} = \mathfrak{p}$, where $\mathfrak{m}$ is the unique maximal ideal of $S$.
	\end{abstract}
	
	\maketitle
	
	\section{Introduction}
	
	Though in different references, the term ``semiring'' is applied for different meanings, in many references (see the explanations on page 3 of the book \cite{Glazek2002}), a semiring is an algebraic structure $(S,+,\cdot,0,1)$ with the following properties:
	
	\begin{enumerate}
		\item $(S,+,0)$ is a commutative monoid,
		\item $(S,\cdot,1)$ is a monoid with $1\neq 0$,
		\item $a(b+c) = ab+ac$ and $(b+c)a = ba+ca$ for all $a,b,c\in S$,
		\item $a\cdot 0 = 0\cdot a = 0$ for all $a\in S$. 
	\end{enumerate}

A semiring $S$ is commutative if $ab = ba$ for all $a,b\in S$. In this paper, all semirings are commutative. 

Semirings are important ring-like structures with many applications in science and engineering \cite[p. 225]{Golan2003} and considered to be interesting generalizations of bounded distributive lattices and commutative rings with nonzero identities \cite[Example 1.5]{Golan1999(b)}. For more on ring-like algebraic structures and their applications, one may refer to \cite{AhsanMordesonShabir2012,Bistarelli2004,Glazek2002,Golan1999(a),Golan1999(b),Golan2003,GondranMinoux2008,HebischWeinert1998,KuichSalomaa1986,Pilz1983}.

A nonempty subset $\mathfrak{a}$ of a semiring $S$ is said to be an ideal of $S$, if $x+y \in \mathfrak{a}$ for all $x,y \in \mathfrak{a}$ and $sx \in \mathfrak{a}$ for all $s \in S$ and $x \in \mathfrak{a}$ \cite{Bourne1951}. We denote the set of all ideals of $S$ by $\Id(S)$. It is easy to verify that $\Id(S)$ with the addition and multiplication of ideals, defined as follows, configures a semiring (check Example 1.4 and Proposition 6.29 in \cite{Golan1999(a)}): \begin{itemize}
	\item $I+J := \{a+b : a\in I, b\in J\},$
	\item $I\cdot J := \{\Sigma^n_{i=1} a_i b_i: a_i \in I, b_i \in J, n\in \mathbb N\}.$
\end{itemize}
An ideal $\mathfrak{a}$ of a semiring $S$ is called a proper ideal of the semiring $S$ if $\mathfrak{a} \neq S$. An ideal $\mathfrak{a}$ of a semiring $S$ is said to be subtractive if $x+y\in \mathfrak{a}$ and $x\in \mathfrak{a}$ imply $y\in \mathfrak{a}$ for all $x,y\in S$. A semiring $S$ is subtractive if each ideal of $S$ is subtractive \cite[\S 6]{Golan1999(b)}. It is clear that any (commutative) ring is subtractive.
	
Let us recall that by definition a proper ideal $\mathfrak{p}$ of a semiring $S$ is prime if $\mathfrak{ab} \subseteq \mathfrak{p}$ implies either $\mathfrak{a} \subseteq \mathfrak{p}$ or $\mathfrak{b} \subseteq \mathfrak{p}$. It is, then, easy to observe that a proper ideal $\mathfrak{p}$ of a semiring $S$ is prime if and only if $ab\in \mathfrak{p}$ implies either $a\in \mathfrak{p}$ or $b\in \mathfrak{p}$ \cite[Corollary 7.6]{Golan1999(b)}. A proper ideal $\mathfrak{m}$ is maximal if $\mathfrak{m} \subseteq \mathfrak{a} \subseteq S$ implies that either $\mathfrak{m} = \mathfrak{a}$ or $\mathfrak{a} = S$ for any ideal $\mathfrak{a}$ of $S$. It is easy to prove that any proper ideal of a semiring is contained in a maximal ideal and any maximal ideal of a semiring is prime. A semiring $S$ is called local if it has only one maximal ideal. It is easy to show that a semiring $S$ is local if and only if $S-U(S)$ is an ideal of $S$, where $U(S)$ is the set of all unit elements of $S$ (see Example 6.1 and Proposition 6.61 in \cite{Golan1999(b)}).
	
We also recall that an ideal $\mathfrak{q}$ of a semiring $S$ is called a primary ideal if $\mathfrak{q}$ is a proper ideal of $S$ and $xy\in \mathfrak{q}$ implies either $x\in \mathfrak{q}$ or $y^n \in \mathfrak{q}$ for some $n\in \mathbb N$ \cite[p. 92]{Golan1999(b)}. It is straightforward to see that if $\mathfrak{q}$ is a primary ideal of a semiring $S$, then $\sqrt\mathfrak{q}$ is the smallest prime ideal containing $\mathfrak{q}$.  In this case, if we set $\mathfrak{p} = \sqrt\mathfrak{q}$, then we say that $\mathfrak{q}$ is $\mathfrak{p}$-primary \cite{Nasehpour2018(a)}. For more on primary ideals of semirings, refer to \cite{Nasehpour2018(a)}.

Note that a semiring $S$ is a semidomain if it is multiplicatively cancellative, i.e. if $xy = xz$, and $x\neq 0$, then $y=z$ for all $x,y,z \in S$. A semiring $S$ is a valuation semiring if it is a semidomain and the set of its ideals $\Id(S)$ is totally ordered by inclusion. For more on valuation semirings, refer to \cite{Nasehpour2018(b)}. In Section \ref{sec:val}, we investigate ideals of valuation semirings.

Let us recall that, in commutative algebra, a prime ideal $\mathfrak{p}$ of a domain $R$ is called divided if in $R_\mathfrak{p}$, we have the following equality: \[\set{z/s \given z\in \mathfrak{p}, s\notin \mathfrak{p}} = \set{z/1 \given z\in \mathfrak{p}}.\] It is easy to verify that a prime ideal $\mathfrak{p}$ of a ring $R$ is a divided prime ideal if and only if $\mathfrak{p} \subset (x)$ for every $x \in R-\mathfrak{p}$. We also note that a domain $D$ is called divided if each prime ideal of $D$ is divided \cite{Dobbs1976}. Divided domains are called AV-domains by Akiba in \cite{Akiba1967}. Inspired by this, we define a prime ideal $\mathfrak{p}$ of a semiring $S$ to be a divided prime ideal of $S$ if $\mathfrak{p} \subset (x)$ for every $x \in S-\mathfrak{p}$ (check Definition \ref{dividedprimeDef}). In Proposition \ref{valuationisdivided}, we prove that every valuation semiring is divided.

We also recall that a proper ideal $\mathfrak{a}$ of a semiring $S$ is called 2-absorbing ideal of $S$ if $xyz\in \mathfrak{a}$ implies either $xy\in \mathfrak{a}$, or $yz \in \mathfrak{a}$, or $xz\in \mathfrak{a}$ \cite[Definition 2.1]{Darani2012}. Section \ref{sec:2-Abs} is devoted to studying 2-absorbing ideals of semirings. For example in Theorem \ref{p2subsetofideal}, we prove that if $\mathfrak{a}$ is a 2-absorbing ideal of a subtractive semiring $S$, then one of the following statements must hold:
\begin{enumerate}
			
		\item $\sqrt{\mathfrak{a}}=\mathfrak{p}$ is a prime ideal of $S$ such that $\mathfrak{p}^2 \subseteq \mathfrak{a}$.
			
		\item $\sqrt{\mathfrak{a}} = \mathfrak{p}_1 \cap \mathfrak{p}_2 $, $\mathfrak{p}_1 \mathfrak{p}_2 \subseteq \mathfrak{a}$ and $(\sqrt{\mathfrak{a}})^2 \subseteq \mathfrak{a}$ where $\mathfrak{p}_1$, $\mathfrak{p}_2$ are the only distinct prime ideals of $S$ that are minimal over $\mathfrak{a}$.
\end{enumerate} 
We also show in Theorem \ref{dividedprimeThm1} that if $\mathfrak{p}$ is a nonzero divided prime ideal of a subtractive semiring $S$ and $\mathfrak{a}$ is an ideal of $S$ such that $\sqrt{\mathfrak{a}}=\mathfrak{p}$, then the following statements are equivalent:
\begin{enumerate}
	\item $\mathfrak{a}$ is a 2-absorbing ideal of $S$;
	\item $\mathfrak{a}$ is a $\mathfrak{p}$-primary ideal of $S$ such that $\mathfrak{p}^2 \subseteq \mathfrak{a}$.
\end{enumerate}

We also prove that if $\mathfrak{p}$ is a nonzero divided prime ideal of a subtractive semidomain $S$, then $\mathfrak{p}^2$ is a 2-absorbing ideal of $S$ (see Theorem \ref{dividedprimeThm2}).

Finally, we add that with the help of the results in Section \ref{sec:val}, we show in Theorem \ref{2-absorbing} that if $S$ is a subtractive valuation semiring and $\mathfrak{a}$ is a nonzero proper ideal of $S$, then the following statements are equivalent:
\begin{enumerate}
	\item $\mathfrak{a}$ is a 2-absorbing ideal of $S$;
	\item $\mathfrak{a}$ is a $\mathfrak{p}$-primary ideal of $S$ such that $\mathfrak{p}^2\subseteq\mathfrak{a}$;
	\item $\mathfrak{a}=\mathfrak{p}$ or $\mathfrak{a}=\mathfrak{p}^2$ where $\mathfrak{p}=\sqrt\mathfrak{a}$ is a prime ideal of $S$.
\end{enumerate}

Note that examples for subtractive valuation semirings include the semiring $\Id(D)$, where $D$ is a Dedekind domain (see Theorem 3.8 and Proposition 3.10 in \cite{Nasehpour2018(b)}).\\

It is obvious that each prime ideal of a semiring is 2-absorbing, though the converse of this statement is not correct. Therefore, it is natural to investigate those semirings in which every 2-absorbing ideal is prime. We define a semiring $S$ to be a 2-AB semiring if every 2-absorbing ideal of $S$ is prime (see Definition \ref{2-ABdef}). We devote Section \ref{sec:2-AB} to 2-AB semirings and in Theorem \ref{2-ABThm2}, we show that if $S$ is a subtractive semiring, then the following statements are equivalent:	
	
	\begin{enumerate}
		
		\item The semiring $S$ is 2-AB.
		\item (a) The prime ideals of $S$ are comparable and (b) if $\mathfrak{p}$ is a minimal prime over a 2-absorbing ideal $\mathfrak{a}$, then $\mathfrak{am} = \mathfrak{p}$, where $\mathfrak{m}$ is the unique maximal ideal of $S$.
		
		\item (a) The prime ideals of $S$ are comparable and (b) for every prime ideal $\mathfrak{p}$ of $S$, $2-\Min_S(\mathfrak{p}^2)=\set{\mathfrak{p}}$.
	\end{enumerate}
Assume that $\mathfrak{a}$ is an ideal of a semiring $S$. We define a 2-absorbing ideal $\mathfrak{p}$ of $S$ to be a minimal 2-absorbing ideal over $\mathfrak{a}$ if there is not a 2-absorbing ideal $\mathfrak{q}$ of $S$ such that $\mathfrak{a}\subseteq\mathfrak{q}\subset\mathfrak{p}$. We denote the set of minimal 2-absorbing ideals over $\mathfrak{a}$ by $2-\Min_S(\mathfrak{a})$ (See Definition \ref{2-MinDef}).

 It is worth mentioning that in Section \ref{sec:2-AB}, we also prove that if $S$ is a subtractive valuation semiring, then, $S$ is 2-AB if and only if $\mathfrak{p}^2 = \mathfrak{p}$ for every prime ideal $\mathfrak{p}$ of $S$ (check Theorem \ref{2-ABThm3}).

Finally, the reader is warned that we use ``$\subseteq$'' for inclusion and ``$\subset$'' for strict inclusion \cite[p. 17]{Monk1969}.

	\section{Ideals of valuation semirings}\label{sec:val}
	
	Let us recall that a semidomain (i.e. multiplicatively cancellative semiring) is a valuation semiring if and only if the ideals of $S$ are totally ordered by inclusion \cite[Theorem 2.4]{Nasehpour2018(b)}. We also recall that by a fractional ideal of a semidomain $S$, we mean a subset $I$ of the semifield of fractions $F(S)$ such that the following conditions are satisfied:
		
		\begin{enumerate}
			
			\item I is an $S$-subsemimodule of $F(S)$, that is, if $a, b \in I$ and $ s \in S$, then $a + b \in I $ and $sa \in I$.
			
			\item There exists a nonzero element $d\in S$ such that $dI \subseteq S$.
			
		\end{enumerate}
	Also, a fractional ideal $I$ of a semidomain $S$ is invertible if there exists a fractional ideal $J$ of $S$ such that $IJ=S$ \cite{GhalandarzadehNasehpourRazavi2017}. It is clear that any invertible ideal $I$ is cancellation, i.e. $IJ =IK$ implies $J=K$, for all ideals $J$ and $K$ of $S$ \cite{LaGrassa1995}. First, we prove the following lemma:
	
	\begin{lemma}
		
		\label{idealsofvaluation1}
		
		Let $S$ be a valuation semiring and $\mathfrak{a}$ a proper ideal of $S$. Then, the following statements hold:
		
		\begin{enumerate}
			
			\item $\bigcap^{\infty}_{n=1} \mathfrak{a}^n = \mathfrak{c}$ is a prime ideal of $S$.
			
			\item If $\mathfrak{b}$ is an ideal of $S$ such that $\mathfrak{a} \subset \sqrt\mathfrak{b}$, then $\mathfrak{b}$ contains a power of $\mathfrak{a}$.
		\end{enumerate}
		
		\begin{proof}
			
			(1): Suppose $s,t\in S$ are not elements of $\bigcap^{\infty}_{n=1} \mathfrak{a}^n = \mathfrak{c}$. Therefore, there are natural numbers $m,n \in \mathbb N$ such that $s \notin \mathfrak{a}^m$ and $t \notin \mathfrak{a}^n$. Since $S$ is a valuation semiring, $\mathfrak{a}^m \subset (s)$ and $\mathfrak{a}^n \subset (t)$. Since $(t)$ is an invertible ideal of $S$ \cite[Proposition 1.4]{GhalandarzadehNasehpourRazavi2017}, $\mathfrak{a}^m (t) \subset (st)$. On the other hand, since $\mathfrak{a}^m \subset (s)$, $\mathfrak{a}^{m+n} \subseteq \mathfrak{a}^n (s)$. So, $\mathfrak{a}^{m+n} \subset (st)$, which means that $st \notin \mathfrak{c}$.
			
			(2): If $\mathfrak{b}$ contains no power of $\mathfrak{a}$, then $\mathfrak{b} \subseteq \mathfrak{a}^n$, for each $n\in \mathbb N$. Therefore, $\mathfrak{b} \subseteq \mathfrak{c}$. Since $\mathfrak{c}$ is prime, $\sqrt\mathfrak{b} \subseteq \mathfrak{c} \subseteq \mathfrak{a}$. This finishes the proof. \end{proof}
	\end{lemma}
	
	\begin{theorem}
		
		\label{idealsofvaluation2}
		
		Let $\mathfrak{p}$ be a proper prime ideal of a valuation semiring $S$.
		\begin{enumerate}
			\item If $\mathfrak{q}$ is $\mathfrak{p}$-primary and $x \in S-\mathfrak{p}$, then $\mathfrak{q}=\mathfrak{q}(x)$.
			\item A finite nonempty product of $\mathfrak{p}$-primary ideals of $S$ is a $\mathfrak{p}$-primary ideal. If $\mathfrak{p} \neq \mathfrak{p}^2$, then the only $\mathfrak{p}$-primary ideals of $S$ are powers of $\mathfrak{p}$.
		\end{enumerate}
		
		\begin{proof}
			(1): As $x \notin \mathfrak{p}$, $\mathfrak{q} \subset (x)$. Suppose that $K$ is the quotient semifield of $S$ and $A=\set{ y \colon y \in K, xy \in \mathfrak{q} }$. As $\mathfrak{q} \subset (x)$, $A$ is a subset of $S$. Now, it is easy to verify that $A$ is, in fact, an ideal of $S$. Our claim is that $\mathfrak{q}=A(x)$. First we prove that $A(x) \subseteq \mathfrak{q}$. Let $s\in A(x)$. Since $A$ is an ideal of $S$, we have that $s =ax$, where $a\in A$. But $s=ax$ is an element of $\mathfrak{q}$ by the definition of $A$. The proof of the other containment is very easy if one considers the definition of $A$ and this point that $\mathfrak{q} \subset (x)$.
			
			Now, we prove that $A \subseteq \mathfrak{q}$. Let $a\in A$. So, $ax\in \mathfrak{q}$, by definition of $A$. Since $\mathfrak{q}$ is $\mathfrak{p}$-primary, we have that either $a\in \mathfrak{q}$ or $x\in \mathfrak{p}=\sqrt{\mathfrak{q}}$. But by assumption, $x \notin \mathfrak{p}$. So, $a\in \mathfrak{q}$ and this shows that $A \subseteq \mathfrak{q}$. Thus $\mathfrak{q}=A$ and $\mathfrak{q}=\mathfrak{q}(x)$.
			\newline
			(2): Suppose that $\mathfrak{q}_1,\mathfrak{q}_2$ are $\mathfrak{p}$-primary ideals of $S$. It is straightforward to see that  $\sqrt{\mathfrak{q}_1\mathfrak{q}_2}=\mathfrak{p}$. Suppose that $x,y \in S$ and $xy \in \mathfrak{q}_1 \mathfrak{q}_2$ and $x \notin \mathfrak{p}$. By the first part of this theorem, we have $\mathfrak{q}_1=\mathfrak{q}_1(x)$. Therefore, $xy \in (x)\mathfrak{q}_1\mathfrak{q}_2$. As $S$ is a semidomain, $y\in \mathfrak{q}_1 \mathfrak{q}_2$. Thus $\mathfrak{q}_1\mathfrak{q}_2$ is $\mathfrak{p}$-primary.
			\newline
			Now let $\mathfrak{p} \neq \mathfrak{p}^2$ and $\mathfrak{q}$ be a $\mathfrak{p}$-primary ideal of $S$. By Lemma \ref{idealsofvaluation1}, $\mathfrak{q}$ contains a power of $\mathfrak{p}^2$ and so contains a power of $\mathfrak{p}$. Thus, there is a positive integer $m$ such that $\mathfrak{p}^m\subseteq\mathfrak{q}$ and $\mathfrak{p}^{m-1} \nsubseteq \mathfrak{q}$. Suppose that $x \in \mathfrak{p}^{m-1}, x \notin \mathfrak{q}$. Clearly, $\mathfrak{q}\subset(x)$. We define $A=\set{ y \given y \in K, xy \in \mathfrak{q} }$, so $\mathfrak{q}=A(x)$. As $\mathfrak{q}$ is $\mathfrak{p}$-primary and $x \notin \mathfrak{q}$, we have $A \subseteq \mathfrak{p}$. Thus, $\mathfrak{q}=A(x)\subseteq\mathfrak{p}(x)\subseteq\mathfrak{p}^m$, and so we conclude that $\mathfrak{q}=\mathfrak{p}^m$.
		\end{proof}
	\end{theorem}

Now, we bring the definition of divided primes and semirings:

\begin{definition}
	
	\label{dividedprimeDef}
	
	We define a prime ideal $\mathfrak{p}$ of a semiring $S$ to be a divided prime ideal of $S$ if $\mathfrak{p} \subset (x)$ for every $x \in S-\mathfrak{p}$. We call a semiring $S$ divided if each prime ideal of $S$ is divided.
\end{definition}

\begin{proposition}
	
	\label{valuationisdivided}
	Any valuation semiring is divided.
	\begin{proof}
		
		Let $S$ be a valuation semiring and $\mathfrak{p}$ be a prime ideal of $S$. Also, let $x\notin \mathfrak{p}$. This means that the principal ideal $(x)$ is not a subset of $\mathfrak{p}$. Since in valuation semirings, ideals are totally ordered by inclusion \cite[Theorem 2.4]{Nasehpour2018(b)}, $\mathfrak{p} \subset (x)$ and the proof is complete.
	\end{proof}
\end{proposition}

One may define the localization of semirings similar to ring theory \cite{Kim1985,Nasehpour2018(a)}. It is, then, straightforward to see that a prime ideal $\mathfrak{p}$ of a semiring $S$ is divided if and only if in $S_\mathfrak{p}$, we have the following equality: \[\set{z/s \given z\in \mathfrak{p}, s\notin \mathfrak{p}} = \set{z/1 \given z\in \mathfrak{p}}.\] By using this point, one can easily prove the following:

\begin{proposition}
	Let $S$ be a divided semiring. Then, each arbitrary ideal of $S$ is comparable with each prime ideal of $S$. In particular, prime ideals of $S$ are comparable and $S$ is local.
\end{proposition}
	
	\section{On 2-Absorbing Ideals of Semirings}\label{sec:2-Abs}
	
The concept of 2-absorbing ideals of semirings was introduced by Darani \cite[Definition 2.1]{Darani2012}. His definition is the semiring version of the 2-absorbing ideals of rings introduced by Badawi \cite{Badawi2007}. We recall this concept in the following:

	\begin{definition}
		
		\label{2-absorbingdef}
		
		A proper ideal $\mathfrak{a}$ of a semiring $S$ is called 2-absorbing ideal of $S$ if $xyz\in \mathfrak{a}$ implies either $xy\in \mathfrak{a}$, or $yz \in \mathfrak{a}$, or $xz\in \mathfrak{a}$.
	\end{definition}

Here is a good place to mention that prime ideals of a semiring are 2-absorbing. Now, we give some other methods to make 2-absorbing ideals:
	
	\begin{proposition}
		
		\label{2-absorbingintersection}
		
		If  $\mathfrak{p}_1$ and $\mathfrak{p}_2$ are both prime ideals of a semiring $S$, then $\mathfrak{p}_1 \cap \mathfrak{p}_2$ is a 2-absorbing ideal of $S$.
		\begin{proof}
			Let $x_1 x_2 x_3 \in \mathfrak{p}_1 \cap \mathfrak{p}_2$ for the elements $x_1, x_2, x_3 \in S$. Then, $x_1 x_2 x_3 \in \mathfrak{p}_1$ and $x_1 x_2 x_3 \in \mathfrak{p}_2$. Since $\mathfrak{p}_1$ is prime, either $x_1 \in \mathfrak{p}_1$, or $x_2 \in \mathfrak{p}_1$, or $x_3 \in \mathfrak{p}_1$. Similarly, one of these three elements is in $\mathfrak{p}_2$. Suppose that $x_1 \in \mathfrak{p}_1$. Now, if $x_1\in \mathfrak{p}_2$, then $x_1 x_2$ is in the both $\mathfrak{p}_1$ and $\mathfrak{p}_2$, so $x_1x_2 \in \mathfrak{p}_1 \cap \mathfrak{p}_2$. If $x_2\in \mathfrak{p}_2$ then $x_1 x_2\in \mathfrak{p}_1 \cap \mathfrak{p}_2$. And if $x_3\in \mathfrak{p}_2$, then $x_1 x_3\in \mathfrak{p}_1 \cap \mathfrak{p}_2$. The same proof is valid for the cases $x_2 \in \mathfrak{p}_1$ and $x_3 \in \mathfrak{p}_1$. Therefore, there are two of these three elements whose product is always in $\mathfrak{p}_1 \cap \mathfrak{p}_2$. Hence, $\mathfrak{p}_1 \cap \mathfrak{p}_2$ is a 2-absorbing ideal of $S$ and the proof is complete.    
		\end{proof}
		
	\end{proposition}
	
	\begin{proposition}
		
		\label{2-absorbingradical}
		
		Let $\mathfrak{a}$ be a 2-absorbing ideal of a semiring $S$. Then, $\sqrt{\mathfrak{a}}$ is a 2-absorbing ideal of $S$ and $s^2 \in \mathfrak{a}$ for each $s\in \sqrt{\mathfrak{a}}$.
		
		\begin{proof}
			
			The proof is just a mimicking of the proof of Theorem 2.1 in \cite{Badawi2007} and therefore, omitted.
		\end{proof}
	\end{proposition}
	
	\begin{proposition}
		
		\label{2-absorbingmulti}
		Let $(S, \mathfrak{m})$ be a local semiring. For each prime ideal $\mathfrak{p}$ of $S$, $\mathfrak{pm}$ is a 2-absorbing ideal of $S$. Furthermore, $\mathfrak{pm}$ is a prime ideal of $S$ if and only if $\mathfrak{pm} = \mathfrak{p}$.
		
		\begin{proof}
			Let $s_1 s_2 s_3 \in \mathfrak{pm}$. Since $\mathfrak{pm} \subseteq \mathfrak{p}$ and $\mathfrak{p}$ is a prime ideal of $S$, we have either $s_1 \in \mathfrak{p}$, or $s_2 \in \mathfrak{p}$, or $s_3 \in \mathfrak{p}$. Now let for the moment $s_1 \in \mathfrak{p}$. If $s_2$ is a unit, then it is obvious that $s^{-1}_2 (s_1 s_2 s_3) = s_1 s_3$ is an element of $\mathfrak{pm}$. If not, then since $S$ is local, $s_2$ is an element of $\mathfrak{m}$. So, $s_1 s_2 \in \mathfrak{pm}$.
			
			Now let $\mathfrak{pm}$ be a prime ideal of $S$. Clearly, $\mathfrak{pm} \subseteq \mathfrak{p}$. Assume $s\in \mathfrak{p}$. It is obvious that $s\in \mathfrak{m}$ and so, $s^2 \in \mathfrak{pm}$. But $\mathfrak{pm}$ is prime, so $s\in \mathfrak{pm}$ and the proof is complete.
		\end{proof} 
	\end{proposition}
	
	Let us recall that a prime ideal $\mathfrak{p}$ of a semiring $S$ is said to be a minimal prime ideal over an ideal $\mathfrak{a}$ of $S$ if it is minimal among all prime ideals containing $\mathfrak{a}$. We collect all minimal prime ideals over an ideal $\mathfrak{a}$ in $\Min(\mathfrak{a})$. 
	
	A subset $W$ of a semiring is called a multiplicatively closed set (for short MC-set) if $1_S \in W$ and if $w_1$ and $w_2$ are elements of $S$, then $w_1 w_2 \in W$. We also recall the following:
	
	\begin{theorem}[Theorem 3.5 in \cite{Nasehpour2018}]
		
		\label{minimalprimehuckaba}
		
		Let $\mathfrak{p} \supseteq \mathfrak{a}$ be ideals of a semiring $S$, where $\mathfrak{p}$ is prime. Then, the following statements are equivalent:
		
		\begin{enumerate}
			
			\item $\mathfrak{p}$ is a minimal prime ideal of $\mathfrak{a}$.
			
			\item $S-\mathfrak{p}$ is an MC-set maximal with respect to being disjoint to $\mathfrak{a}$.
			
			\item For each $x\in \mathfrak{p}$, there is some $y\in S-\mathfrak{p}$ and a nonnegative integer $i$ such that $yx^i \in \mathfrak{a}$.
			
		\end{enumerate}
		
	\end{theorem}
	
	\begin{theorem}
		
		\label{minimalprimeover2-ab}
		Let $\mathfrak{a}$ be a 2-absorbing ideal of a semiring $S$ such that each minimal prime ideal over $\mathfrak{a}$ is subtractive. Then, $|\Min(\mathfrak{a})| \leq 2$.
		
		\begin{proof}
			
			By considering Theorem \ref{minimalprimehuckaba} and this point that each minimal prime ideal over the ideal $\mathfrak{a}$ is subtractive, the proof is nothing but a mimicking of the proof of Theorem 2.3 in \cite{Badawi2007}.
		\end{proof}
		
	\end{theorem}
	
	Let us recall that, by definition, a semiring $S$ is weak Gaussian if $c(f)c(g) \subseteq \sqrt {c(fg)} $ for all $f,g \in S[X]$, where the content of a polynomial $f\in S[X]$, denoted by $c(f)$, is an ideal of $S$ generated by the coefficients of $f$. Note that a semiring $S$ is weak Gaussian if and only if each prime ideal of $S$ is subtractive \cite[Definition 18, Theorem 19]{Nasehpour2016}.
	
	\begin{corollary}
		Let $S$ be a weak Gaussian semiring and $\mathfrak{a}$ a 2-absorbing ideal of $S$. Then, $|\Min(\mathfrak{a})| \leq 2$.
	\end{corollary}
	
	\begin{theorem}
		
		\label{p2subsetofideal}
		
		Let $\mathfrak{a}$ be a 2-absorbing ideal of a subtractive semiring $S$. Then, one of the following statements must hold:
		
		\begin{enumerate}
			
			\item $\sqrt{\mathfrak{a}}=\mathfrak{p}$ is a prime ideal of $S$ such that $\mathfrak{p}^2 \subseteq \mathfrak{a}$.
			
			\item $\sqrt{\mathfrak{a}} = \mathfrak{p}_1 \cap \mathfrak{p}_2 $, $\mathfrak{p}_1 \mathfrak{p}_2 \subseteq \mathfrak{a}$ and $(\sqrt{\mathfrak{a}})^2 \subseteq \mathfrak{a}$ where $\mathfrak{p}_1$, $\mathfrak{p}_2$ are the only distinct prime ideals of $S$ that are minimal over $\mathfrak{a}$.
		\end{enumerate}
		
	\end{theorem}
	\begin{proof}
		By considering Theorem \ref{minimalprimeover2-ab} and this point that $S$ is a subtractive semiring, the proof is nothing but a mimicking of the proof of Theorem 2.4 in \cite{Badawi2007}.
	\end{proof}

	\begin{theorem}
		\label{m2twoabsorbing}
		Let $\mathfrak{a}$ be a $\mathfrak{p}$-primary ideal of a subtractive semiring $S$. Then, $\mathfrak{a}$ is a 2-absorbing ideal of $S$ if and only if $\mathfrak{p}^2 \subseteq \mathfrak{a}$. In particular, for each maximal ideal $\mathfrak{m}$ of $S$, $\mathfrak{m}^2$ is a 2-absorbing ideal of $S$.
		
		\begin{proof}
			Let $\mathfrak{a}$ be a 2-absorbing ideal of a subtractive semiring $S$. Then, by Theorem \ref{p2subsetofideal}, we have $\mathfrak{p}^2 \subseteq \mathfrak{a}$. Conversely, let $xyz \in \mathfrak{a}$ and $\mathfrak{p}^2 \subseteq \mathfrak{a}$. If $x \in \mathfrak{a}$ or $yz \in \mathfrak{a}$, then there is nothing to prove. If neither $x \in \mathfrak{a}$ nor $yz \in \mathfrak{a}$, as $\mathfrak{a}$ is a $\mathfrak{p}$-primary ideal of $S$, then $x \in \mathfrak{p}$ and $yz \in \mathfrak{p}$. Therefore, either $x,y \in \mathfrak{p}$ or $x,z \in \mathfrak{p}$. Now from the assumption $\mathfrak{p}^2 \subseteq \mathfrak{a}$, we get that either $xy \in \mathfrak{a}$ or $xz \in \mathfrak{a}$ and this completes the proof.
		\end{proof}
	\end{theorem}

	\begin{theorem}
		
		\label{dividedprimeThm1}
		
		Let $\mathfrak{p}$ be a nonzero divided prime ideal of a subtractive semiring $S$ and $\mathfrak{a}$ be an ideal of $S$ such that $\sqrt{\mathfrak{a}}=\mathfrak{p}$. Then, the following statements are equivalent:
		\begin{enumerate}
			\item $\mathfrak{a}$ is a 2-absorbing ideal of $S$;
			\item $\mathfrak{a}$ is a $\mathfrak{p}$-primary ideal of $S$ such that $\mathfrak{p}^2 \subseteq \mathfrak{a}$.
		\end{enumerate}
	\end{theorem}
	\begin{proof}
		(1) $\Rightarrow$ (2): Let $\mathfrak{a}$ be a 2-absorbing ideal of $S$. By using Theorem \ref{p2subsetofideal} and considering this point that $\sqrt{\mathfrak{a}}=\mathfrak{p}$ is a nonzero prime ideal of $S$, we have $\mathfrak{p}^2 \subseteq \mathfrak{a}$. Suppose that for some $x,y \in S$, we have $xy \in \mathfrak{a}$ and $y \notin \mathfrak{p}$. As $\mathfrak{p}$ is a divided prime ideal of $S$ and $x \in \mathfrak{p}$, we conclude that $x=my$ for some $m \in S$. Therefore, $xy=my^2 \in \mathfrak{a}$. Since $y^2 \notin \mathfrak{a}$ and $\mathfrak{a}$ is a 2-absorbing ideal of $S$, we have $my=x \in \mathfrak{a}$. Thus, $\mathfrak{a}$ is a $\mathfrak{p}$-primary ideal of $S$.
		\newline
		(2) $\Rightarrow$ (1): Using Theorem \ref{m2twoabsorbing}, the proof of this implication is obvious.
	\end{proof}
	
	\begin{theorem}
		
		\label{dividedprimeThm2}
		
		Let $\mathfrak{p}$ be a nonzero divided prime ideal of a subtractive semidomain $S$. Then, $\mathfrak{p}^2$ is a 2-absorbing ideal of $S$.
		
		\begin{proof}
			
			By Theorem \ref{dividedprimeThm1}, we only need to show that $\mathfrak{p}^2$ is a $\mathfrak{p}$-primary ideal of $S$. Let $st\in \mathfrak{p}^2$, while $s\notin \mathfrak{p}$. Suppose that $st = \Sigma^n_{i=1} x_i y_i$, where $x_i,y_i \in \mathfrak{p}$. Since $\mathfrak{p}$ is a divided prime ideal of $S$ and $s\notin \mathfrak{p}$, we have $\mathfrak{p} \subset (s)$. Now since $x_i \in \mathfrak{p}$, we have $x_i = s z_i$, for each $1\leq i \leq n$. On the other hand, since $s\notin \mathfrak{p}$, we have $z_i \in \mathfrak{p}$. Now $st = s(\Sigma^n_{i=1} z_i y_i)$. By assumption $S$ is multiplicatively cancellative. So, $t = \Sigma^n_{i=1} z_i y_i$, which means that $t\in \mathfrak{p}^2$. Hence, $\mathfrak{p}^2$ is $\mathfrak{p}$-primary and the proof is complete.
		\end{proof}
		
	\end{theorem}
	
	\begin{question}
    Is there any example of a semiring $S$ that contains a prime ideal $\mathfrak{p}$ for which $\mathfrak{p}^2$ is not 2-absorbing?
	\end{question}

     Related to the above question, we also invite the reader to check Corollary \ref{not-2-absorbing}. Now we proceed to investigate 2-absorbing ideals of valuation semirings. For doing this, we need to prove some statements for the ideals of valuation semirings. 
	
	\begin{theorem}
		
		\label{2-absorbing}
		
		Let $S$ be a subtractive valuation semiring and $\mathfrak{a}$ be a nonzero proper ideal of $S$. Then, the following statements are equivalent:
		\begin{enumerate}
			\item $\mathfrak{a}$ is a 2-absorbing ideal of $S$;
			\item $\mathfrak{a}$ is a $\mathfrak{p}$-primary ideal of $S$ such that $\mathfrak{p}^2\subseteq\mathfrak{a}$ where $\mathfrak{p} = \sqrt{\mathfrak{a}}$ is a prime ideal of $S$;
			\item $\mathfrak{a}=\mathfrak{p}$ or $\mathfrak{a}=\mathfrak{p}^2$ where $\mathfrak{p}=\sqrt\mathfrak{a}$ is a prime ideal of $S$.
		\end{enumerate}
	
	\begin{proof}
		(1) $\Rightarrow$ (2): Let $\mathfrak{a}$ be a 2-absorbing ideal of $S$. It is easy to verify that $\sqrt\mathfrak{a}=\mathfrak{p}$ is a prime ideal of $S$. Since $S$ is a valuation semiring, by Proposition \ref{valuationisdivided}, $S$ is a divided semidomain. Now by using Theorem \ref{dividedprimeThm1}, $\mathfrak{a}$ is a $\mathfrak{p}$-primary ideal of $S$ such that $\mathfrak{p}^2\subseteq\mathfrak{a}$. 
		\newline
		(2) $\Rightarrow$ (3): Let $\mathfrak{a}$ be a $\mathfrak{p}$-primary ideal of $S$ such that $\mathfrak{p}^2 \subseteq \mathfrak{a}$. As $S$ is a valuation semiring, by using Theorem \ref{idealsofvaluation2} we can conclude that either $\mathfrak{a}=\mathfrak{p}$ or $\mathfrak{a}=\mathfrak{p}^2$.
		\newline
		(3) $\Rightarrow$ (1): Suppose that either $\mathfrak{a}=\mathfrak{p}$ or $\mathfrak{a}=\mathfrak{p}^2$  where $\mathfrak{p}=\sqrt{\mathfrak{a}}$ is a prime ideal of $S$. If $\mathfrak{a}=\mathfrak{p}$, then $\mathfrak{a}$ is a 2-absorbing ideal of $S$. If $\mathfrak{a}=\mathfrak{p}^2$, then by using Theorem \ref{dividedprimeThm2}, $\mathfrak{a}$ is a 2-absorbing ideal of $S$.  
	\end{proof}

\end{theorem}
	
	\section{2-AB semirings}\label{sec:2-AB}
	
	Let us recall that a ring $R$ is called 2-AB if every 2-absorbing ideal of $R$ is prime \cite[Definition 2.1]{BennisFahid2017}. Inspired by this, we give the following definition:
	
	\begin{definition}
		
		\label{2-ABdef}
		
		We define a semiring $S$ to be a 2-AB semiring if every 2-absorbing ideal of $S$ is prime.
	\end{definition}

Let us recall that a 2-absorbing ideal $\mathfrak{p}$ of a ring $R$ is said to be a minimal 2-absorbing ideal over an ideal $\mathfrak{a}$ of $R$ if it is minimal among all 2-absorbing ideals containing $\mathfrak{a}$ \cite{MoghimiNaghani2016}. The following is the semiring version of the definition of minimal 2-absorbing ideals in commutative rings:

\begin{definition}
	
	\label{2-MinDef}
	
	Let $\mathfrak{a}$ be an ideal of a semiring $S$. We define a 2-absorbing ideal $\mathfrak{p}$ of $S$ to be a minimal 2-absorbing ideal over $\mathfrak{a}$ if there is not a 2-absorbing ideal $\mathfrak{q}$ of $S$ such that $\mathfrak{a}\subseteq\mathfrak{q}\subset\mathfrak{p}$. We denote the set of minimal 2-absorbing ideals over $\mathfrak{a}$ by $2-\Min_S(\mathfrak{a})$.
\end{definition}

\begin{lemma}
	
	\label{2minsemiring}
	
	Let $S$ be a subtractive semiring such that the prime ideals of $S$ are comparable and $\mathfrak{m}$ be the unique maximal ideal of $S$. Then, the following statements are equivalent:
	
	\begin{enumerate}
		
		\item For every minimal prime ideal $\mathfrak{p}$ over a 2-absorbing ideal $\mathfrak{a}$, $\mathfrak{am}=\mathfrak{p}$;
		
		\item For every prime ideal $\mathfrak{p}$ of $S$, $2-\Min_S(\mathfrak{p}^2)=\set{\mathfrak{p}}$.
		
	\end{enumerate}

    Moreover, if one of the above equivalent statements holds, then $S$ is 2-AB.
	
	\begin{proof}
		$(1) \Rightarrow (2)$: Let $\mathfrak{p}$ be a prime ideal of $S$ and $\mathfrak{b}$ be a 2-absorbing ideal of $S$ such that $\mathfrak{b}\in 2-\Min_S(\mathfrak{p}^2)$. First, we prove that $\mathfrak{p} \in \Min_S(\mathfrak{b})$. Suppose there exists a prime ideal $\mathfrak{q}$ such that $\mathfrak{b} \subseteq \mathfrak{q} \subseteq \mathfrak{p}$. Clearly, $\mathfrak{p}^2 \subseteq \mathfrak{b} \subseteq \mathfrak{q} \subseteq \mathfrak{p}$. Let $x\in \mathfrak{p}$. So, we have $x^2 \in \mathfrak{p}^2$ and therefore, $x^2 \in \mathfrak{q}$. Since $\mathfrak{q}$ is prime, $x \in \mathfrak{q}$. Thus, $\mathfrak{p}=\mathfrak{q}$. But $\mathfrak{b}\mathfrak{m} \subseteq \mathfrak{b} \subseteq \mathfrak{p}$. Now since by hypothesis $\mathfrak{b}\mathfrak{m}=\mathfrak{p}$, we have $\mathfrak{b}=\mathfrak{p}$.
		\newline
		$(2) \Rightarrow (1)$: Suppose that $\mathfrak{p}$ is a minimal prime ideal over a 2-absorbing ideal $\mathfrak{a}$. Since $\sqrt\mathfrak{a}$ is a prime ideal of $S$ and $\sqrt\mathfrak{a} \subseteq \mathfrak{p}$, we have $\sqrt\mathfrak{a} = \mathfrak{p}$. So, by Theorem \ref{p2subsetofideal}, $\mathfrak{p}^2 \subseteq \mathfrak{a} \subseteq \mathfrak{p}$. Now by hypothesis $2-\Min_S(\mathfrak{p}^2)=\set{\mathfrak{p}}$. Therefore, $\mathfrak{a}=\mathfrak{p}$. By Proposition \ref{2-absorbingmulti}, $\mathfrak{pm}$ is 2-absorbing. Now, $\mathfrak{p}^2 \subseteq \mathfrak{pm} \subseteq \mathfrak{p}$. So, $\mathfrak{am}=\mathfrak{pm}=\mathfrak{p}$.
		
		Now, let the statement (1) hold. If $\mathfrak{a}$ is a 2-absorbing ideal of $S$, then there is a minimal prime ideal $\mathfrak{p}$ over $\mathfrak{a}$ and clearly, we have $\mathfrak{am} \subseteq \mathfrak{a} \subseteq \mathfrak{p}$, which implies that $\mathfrak{a} = \mathfrak{p}$.
	\end{proof}
\end{lemma}

\begin{corollary}
	Let $S$ be a subtractive divided semidomain with the unique maximal ideal $\mathfrak{m}$. Then, the following statements are equivalent:
	
	\begin{enumerate}
		
		\item For every minimal prime ideal $\mathfrak{p}$ over a 2-absorbing ideal $\mathfrak{a}$, $\mathfrak{am}=\mathfrak{p}$;
		
		\item For every prime ideal $\mathfrak{p}$ of $S$, $2-\Min_S(\mathfrak{p}^2)=\set{\mathfrak{p}}$.
		
	\end{enumerate} 
Moreover, if one of the above equivalent statements holds, then $S$ is 2-AB.
\end{corollary}
	
	Now, we prove the following lemma:
	
	\begin{lemma}
		
		\label{2-ABThm1}
		
		Let $S$ be a 2-AB semiring. Then, the following statements hold:
		
		\begin{enumerate}
			\item Prime ideals of $S$ are comparable; in particular, $S$ is a local semiring.
			
			\item Let $\mathfrak{m}$ be the unique maximal ideal of $S$. If $\mathfrak{p}$ is a minimal prime over a 2-absorbing ideal $\mathfrak{a}$, then $\mathfrak{am} = \mathfrak{p}$. In particular, $\mathfrak{m}$ is an idempotent ideal of $S$.
		\end{enumerate}
		
		\begin{proof}
			(1): Let $\mathfrak{p}_1$ and $\mathfrak{p}_2$ be prime ideals of $S$. By Proposition \ref{2-absorbingintersection}, $\mathfrak{p}_1 \cap \mathfrak{p}_2$ is a 2-absorbing ideal of $S$ and since $S$ is 2-AB, $\mathfrak{p} = \mathfrak{p}_1 \cap \mathfrak{p}_2$ is prime. Therefore, either $\mathfrak{p}_1 = \mathfrak{p}_1 \cap \mathfrak{p}_2$ or $\mathfrak{p}_2 = \mathfrak{p}_1 \cap \mathfrak{p}_2$. This means that prime ideals of $S$ are comparable and so, $S$ is local.
			
			(2): Let $\mathfrak{p}$ be a minimal prime over a 2-absorbing ideal $\mathfrak{a}$. Since $S$ is 2-AB, $\mathfrak{a}$ is prime and so, $\mathfrak{p} = \mathfrak{a}$. Now, by Proposition \ref{2-absorbingmulti}, $\mathfrak{am} = \mathfrak{p}$ and this finishes the proof.
		\end{proof}
	\end{lemma}

	\begin{theorem}
		
		\label{2-ABThm2}
		
		Let $S$ be a subtractive semiring. Then the following statements are equivalent:

		\begin{enumerate}
			
			\item \label{2-ABThm2-1} The semiring $S$ is 2-AB.
			\item \label{2-ABThm2-2} (a) The prime ideals of $S$ are comparable and (b) if $\mathfrak{p}$ is a minimal prime over a 2-absorbing ideal $\mathfrak{a}$, then $\mathfrak{am} = \mathfrak{p}$, where $\mathfrak{m}$ is the unique maximal ideal of $S$.
			
			\item \label{2-ABThm2-3} (a) The prime ideals of $S$ are comparable and (b) for every prime ideal $\mathfrak{p}$ of $S$, $2-\Min_S(\mathfrak{p}^2)=\set{\mathfrak{p}}$.
		\end{enumerate}
		
		\begin{proof}
			
			According to Lemma \ref{2-ABThm1}, $(1) \Rightarrow (2)$. For the proof of $(2) \Rightarrow (1)$, let $\mathfrak{a}$ be a 2-absorbing ideal of $S$. By assumption, prime ideals of $S$ are comparable. Therefore, by Theorem \ref{p2subsetofideal}, we have $\mathfrak{p}^2 \subseteq \mathfrak{a} \subseteq \mathfrak{p}$, where $\mathfrak{p}$ is a minimal prime over $\mathfrak{a}$. Now by assumption, $\mathfrak{am} = \mathfrak{p}$. So, $\mathfrak{p} \subseteq \mathfrak{a} \cap \mathfrak{m} = \mathfrak{a}$ and the proof is complete.
		
	By Lemma \ref{2minsemiring} and this point that the statements (1) and (2) are equivalent, clearly, the statements (1) and (3) are equivalent and this finishes the proof.\end{proof}
	\end{theorem}
	
	\begin{corollary}
		
		\label{not-2-absorbing}
		
		Let $S$ be a subtractive 2-AB semiring. For each prime ideal $\mathfrak{p}$ of $S$, either $\mathfrak{p} = \mathfrak{p}^2$ or $\mathfrak{p}^2$ is not 2-absorbing.
	\end{corollary}

\begin{question}
	Is it true that a subtractive semiring $S$ is 2-AB if primes are comparable and if $\mathfrak{pm} = \mathfrak{p}$ for all prime ideals $\mathfrak{p}$?
\end{question}
	
	\begin{remark}
		Let $R$ be a commutative ring with a nonzero identity. By Nakayama's lemma in commutative algebra and Proposition \ref{2-absorbingmulti}, it is straightforward to see that if $R$ is a Noetherian 2-AB ring, then $R$ is a field. Now let $S$ be a Noetherian 2-AB semiring. One may ask if $S$ is a semifield. In fact, this is not the case. Consider the semiring $S=\{0,u,1\}$, where $1+u = u+1 = u$, $1+1=1$, and $u+u = u\cdot u =u$ \cite{LaGrassa1995}. The only ideals of $S$ are $(0)$, $S$, and $\{0,u\}$. Note that the ideal $\{0,u\}$ is maximal (and prime). Also, it is clear that $(0)$ is prime. Therefore, $S$ is 2-AB. Clearly, $S$ is Noetherian, but it is not a semifield, because the element $u$ is not a unit.
	\end{remark}

\begin{theorem}
	
	\label{2-ABThm3}

Let $S$ be a subtractive valuation semiring. Then, $S$ is 2-AB if and only if $\mathfrak{p}^2 = \mathfrak{p}$ for every prime ideal $\mathfrak{p}$ of $S$.
	
\begin{proof}
$(\Rightarrow)$: Suppose that $S$ is 2-AB. Since $S$ is a subtractive valuation semiring, $\mathfrak{p}^2$ is a $\mathfrak{p}$-primary ideal of $S$. Therefore, by Theorem \ref{2-absorbing}, $\mathfrak{p}^2$ is 2-absorbing and so by hypothesis, a prime ideal of $S$. This implies that $\mathfrak{p}^2 = \mathfrak{p}$.
		
$(\Leftarrow)$: This is obvious by Theorem \ref{2-absorbing}.
\end{proof}
\end{theorem}

\begin{remark}
In some parts of our paper, we suppose a semiring to be a (subtractive) valuation semiring. On the other hand, in some theorems, we suppose a semiring to be divided. Note that any valuation semiring is divided. Now, one may ask if there are some subtractive valuation semirings which are not rings and so we have really generalized some results in commutative algebra. We point out that examples for subtractive valuation semirings include the semiring $\Id(D)$, where $D$ is a Dedekind domain (refer to Theorem 3.8 and Proposition 3.10 in \cite{Nasehpour2018(b)}). 
\end{remark}

\subsection*{Acknowledgments}

The second named author is supported by the Department of Engineering Science at the Golpayegan University of Technology and his special thanks go to the Department for providing all necessary facilities available to him for successfully conducting this research. The authors are also grateful to the esteemed referee for her/his through reading and review which improved the paper.

\end{document}